\documentclass[article,noinfoline]{imsart}
\usepackage{graphicx,enumerate}

\usepackage{jr,comment}

\newtheorem{Assumption}{Assumption}
\startlocaldefs
\def\bb#1{\mathbb{#1}}

\def\boldsymbol#1{\setbox\ewb\hbox{$#1$}%
    \setlength{\deno}{-\wd\ewb+0.05em}{ #1}\hspace{\deno}{#1}}
\endlocaldefs

\begin{document}

\begin{frontmatter}

\title{Generalized Maximum Likelihood Estimators 
and their applications 
to stratified sampling and post-stratification with many  unobserved strata.}

\author{\fnms{Eitan} \snm{Greenshtein}\ead[label=e2]{eitan.greenshtein@gmail.com}}
\address{Israel Census Bureau of Statistics; \printead{e2}}
\affiliation{Israel Census Bureau of Statistics}

\author{\fnms{Ya'acov} \snm{Ritov}
\ead[label=e3]{yaacov.ritov@gmail.com}}
\address{University of Michigan; \printead{e3}}
\affiliation{University of Michigan.}

\runauthor{ Greenshtein,   Ritov}

\maketitle

\begin{abstract}

Consider the problem of estimating  a weighted average of the means of $n$ strata, based on a random sample  with realized
$K_i$ observations from stratum $i, \; i=1,...,n$. 

This task is non-trivial  in cases where for  a significant portion of the strata the corresponding $K_i=0$.
Such a situation may happen in post-stratification, when it is desired to have a very fine sftratification.
A fine stratification could be desired in order that assumptions, or, approximations, like  Missing At Random conditional on strata, will be appealing.
A fine stratification could also be desired in  observational studies, when it is desired to estimate average treatment effect, by averaging 
the effects in small and homogenous strata.

Our approach is based  on applying Generalized Maximum Likelihood Estimators (GMLE), and ideas that are related
to Non-Parametric Empirical Bayes, in order to estimate the means of strata  $i$ with corresponding $K_i=0$. There are no assumptions
about  a relation between the means of the unobserved strata (i.e., with $K_i=0$) and those of the observed strata.

The performance of our approach is demonstrated both in simulations and  on a real data set.
Some consistency and asymptotic  results are also presented. In addition, related basic results
about GMLE estimation of the mean of mixtures of exponential families are provided.

\end{abstract}

\end{frontmatter}

\section{Introduction}

In this paper we study the problem of estimating population's average
when the sampling distribution is not entirely known to the statistician,
and Horwitz-Thompson estimators can not be applied.

An example of such a situation is  non-response, see, e.g., Little and Rubin (2002),
where response probabilities of individuals
are not known and hence neither their sampling probabilities. Here, the  sampling probability
of an individual  refers to the probability of obtaining the corresponding value of interest
and, in particular, getting a response from that individual. 
An approach to treat non-response in a different setup, is Greenshtein and
Itskov (2018).  In that paper they  apply GMLE ideas  to estimate  the unknown response probabilities.

Another situation where sampling probabilities are not entirely known is  Observational Studies,
see, e.g., Rosenbaum and Rubin (1983) and Robins and Ritov (1997).  Consider for example a situation
where it is desired to estimate the `average treatment effect', where treatment is vaccination,
and vaccination is not mandatory. Then, the probability of an individual  to get a vaccination
is not based on the experimenter assignment and it is not known. 

Our approach, as most other approaches, relies on the `{\it hope}' that within delicate enough strata
(or, conditional on a fine stratification), the effect of the sampling distribution becomes negligible, as in
`Missing At Random"' (MAR) assumption, or  `Strongly Ignorable Assignment' assumption. Of course,
such assumptions are typically only `approximately right', yet, they become more appealing as the strata become finer.

On one hand the negligibility of the sampling distribution conditional on strata, has more appeal, as the strata become finer. On the other hand, under very fine stratification, there will be many unobserved strata, which creates difficulty in estimating the population's average. Our, novel contribution is to suggest GMLE approach to handle unobserved strata and yield estimators for
their corresponding strata's average, and consequently to
the population's average. 

Let $K_i$ be the distribution  of the random size of the number of observations from stratum $i$.
Our GMLE approach is for situations where one may reasonably model parametrically the  distribution
of $K_i$.

\subsection{ Motivation and Triangular Array Formulation and Asymptotics.}

Initially, it is convenient to describe and motivate our approach through strata
with equal proportions, i.e., $n$ strata  with (known) equal proportions $1/n$ in the population. Later, in Section \ref{sec:gen} we generalize to unequal proportions.

It is simpler to motivate and explain our ideas in light of binary variables. 
Specifically, it is desired to estimate a population's proportion $p=\frac{1}{n} \sum p_i$, where $p_i$ 
 is the population's proportion in Strata $i$.
In Section \ref{sec:gen}
we will generalize to non-binary variables.

\bigskip

{\bf Non-Response}

Consider a setup of non-response, where the response 
probability 
of 
a random subject from
stratum $i$ 
is 
denoted $\pi_i$, $\pi_i<1$. Here $\pi_i$ does not necessarily equal to
$\pi_{ij}$--the response probability of  individual $j$ in 
Stratum $i$. It is planned to sample $\kappa_i$ units from Stratum $i$,
however, the number of responders is $K_i \leq \kappa_i$.

In order to fix ideas, suppose that it is desired to  estimate the proportion of unemployed in the population. It is known that in the relevant surveys, neighborhoods
(or, more generally, strata) with higher unemployment rates have lower response
rates.

Let $\alpha_i$ be the proportion of unemployed in Stratum $i$, 
while $p_i$ is the  expected proportion among sampled  responders  
from Stratum $i$.
Formally, let $I$ be an indicator of the event that a randomly sampled
person is unemployed, let $R$ denote the event that the person responded, and let $S_i$
be the event that the sampled person is in Stratum $i$. Then:
$$\alpha_i=E(I| S_i), \;\;\; p_i= E(I|S_i, R).$$ 

In general $\alpha_i \neq p_i$. 
Equality follows under 
`Missing At Random' (MAR)  conditional on the strata.
The quantity of interest is $\alpha= \frac{1}{n} \sum \alpha_i$.
The motivation and effort in this paper, for the estimation of ${p}=\frac{1}{n} \sum p_i$ is 
through `approximate MAR assumption',
under which ${p_i} \approx {\alpha_i}$, $i=1,...,n$.

The later has more appeal as the  stratification becomes finer.

In this setup, modeling the number of observations $K_i$, obtained from
Stratum $i$,  as Binomial, is reasonable. 

\bigskip

{\bf Observational Study and Causal Effect.}

Suppose we want to study the effect of 
receiving a vaccine (treatment)  on the probability 
of catching a certain disease.
If vaccination is not mandatory, the number of 
people, $K_i$, that `voluntarily' receive  the treatment from stratum $i$, 
$i=1,...,n$, is random.
Moreover, the probability of an individual to voluntarily get the 
treatment  is not really known.
In such an observational study,  modeling  
$K_i$
as Poisson is reasonable. 

Let $\alpha_i$ be the probability of catching the disease for a random subject from Stratum $i$, that receives the vaccine. Let $p_i$ be the probability of catching the 
disease, for a random person from Stratum $i$ in the observational study, that 
voluntarily receives the
vaccine.  
Let $I$ be an indicator of the event the person caught the disease; let $S_i$ be as before. Let $A$ denote the event that the person was vaccinated, let $B$ denote
the event that the person  was `voluntarily' vaccinated. Here $A$ reflects an 
event in a potential designed experiment,
while $B$ is an event in the actual observational study. Formally,

$$\alpha_i= E(I| S_i, A), \;\; p_i=E(I| S_i, B).$$

In general $p_i \neq \alpha_i$. Equality is implied under assumptions of the nature of
`Strongly Ignorable Assignment'.

The quantity of interest is $\alpha= \frac{1}{n} \sum \alpha_i$.
Our motivation to estimate $p=\frac{1}{n} \sum_i p_i$ 
is based on the `hope' that $\alpha_i \approx p_i$. The later `hope' appeals more as the stratification  becomes finer.

The above approach is readily extended  to comparison between two treatments, by simply
estimating the corresponding $p_{ij}$ separately for treatment $j$, $j=1,2$. This is conceptually different than some ideas that use propensity score, where it is desired to make a retrospect matching between observations that received different treatments, based on their propensity score.

 \bigskip

{\it Summary.} The above discussion and the distinction between $\alpha$ and $p$, is brought in order to motivate the estimation of $p$ rather than $\alpha$, since the former estimation is practical. In addition it motivates the very fine stratification where consequently many of the strata are unobserved, i.e., with $K_i=0$.  In the following sections we concentrate on the problem of estimating $p$, and {\it not} on the problem  of estimating $\alpha$.

As elaborated in the following, when we have more observations the stratification becomes finer, and thus, even in an asymptotic sense as the number of observations approaches infinity,
there are many unobserved strata.

\bigskip

\bigskip

{\bf Triangular array Parametrization.}

As fore mentioned, when the number of observations $m$ increases, we consider different, typically finer, stratification. Let $n=n(m)$ be the number of strata in our stratification.
The choice of a specific stratification and in particular the relation $n=n(m)$ is beyond the scope of this paper.

Given a sequence of stratifications with $n$ strata, consider  the corresponding
$p_i^n$ and $\alpha_i^n$. Let  $ p^n= \frac{1}{n} \sum p_i^n$ and $\alpha^n=\frac{1}{n} \sum \alpha_i^n$. Then, $\alpha^1=\alpha^2=...\equiv\alpha$, while typically $p^1\neq p^2 \neq ...$.
As explained, we `hope' that for large $n$ (or, under fine stratification) $p^n$ becomes closer to 
$\alpha$.  Our approach of estimating a sequence of parameters $p^n$ in a sequence of problems
is in the spirit of `triangular array', see Greenshtein and Ritov (2004). 
Their setup involves prediction rather then estimation in this paper, see also Greenshtein 
(2006).
In particular, the goal at stage $n$ is to estimate $p^n$ by some $\hat{p}^n$, thus, 
the target parameter $p^n$ is changed with $n$.
In the spirit
of triangular array, in our estimation problem, given a sequence of stratifications, we say that $\hat{p}^n$
has  a persistence property iff for any sequence of distributions,
$$\hat{p}^n - p^n \rightarrow_W 0;$$ here we consider weak convergence, or, convergence in probability.
Note, there is no assumption that $p^n \rightarrow \alpha$; if the later is satisfied  then  persistence is just consistency in the estimation of 
$\alpha$. 
In persistence, we only require that the estimation of the parameter $p^n$ is `decently' done for large $n$. The later indicates that the stratification is not too delicate.

In the sequel, we neglect the above formalities, and just right $p$ rather than $p^n$.
In our GMLE formulation we embed the $n'th$ problem in an auxiliary asymptotic problem
in which we may address the issue of consistency in estimating $p^n$, rather than formal persistence. We elaborate on the triangular array formulation and on persistence in Sub-Section \ref{rate}, 
but  it is mostly neglected in the sequel.

\subsection{ Random sample sizes scenarios, and generalizations. }

As aforementioned, the following are two realistic  scenarios, 
where the sample sizes $K_i$ are random.

Scenario i) {\it Stratified sampling with non-response.}
Consider a situation where  it is planned to randomly 
sample $\kappa_i$  subjects from stratum $i$. However, 
the probability of a random  subject from stratum $i$ to 
respond is $\pi_i \leq 1$.   
The number of actual responses, $K_i$, is reasonably 
modeled, 
for large strata, as  distributed 
$K_i \sim B(\kappa_i, \pi_i)$.

Scenario ii): {\it  Post Stratification}. 
We conduct a random sample 
from a population, 
let $K_i$ be the number of  responded subjects from 
stratum
$i$, $i=1,...,n$. 
When $n$ is large and strata are small, it is 
reasonable to model $K_i$ as distributed  $Poisson(\lambda_i)$. 
This  is reasonable both, under probability $\pi_i=1$  
and  $\pi_i<1$, of response from a random subject 
from stratum $i$. This scenario appeals also in observational studies,
as previously discussed.

\bigskip

{\bf Generalizations.} 

\bigskip

In Section \ref{sec:gen} we will also consider the more general 
problem 
of estimating $\sum a_i p_i$ for  given $a_1,...,a_n$, 
where $a_i \equiv 1/n$ does not necessarily  hold. 
Consider, for example 
the case where $a_i$ 
is the known proportion of 
stratum $i$ 
in 
the population, and strata are not of equal sizes. 

\bigskip

We will further generalize to the problem of estimating 
\begin{equation} \label{eqn:general} \sum a_i \mu_i, \end{equation} 

where $\mu_i=E \frac{X_i}{K_i}$, and $X_i=\sum_{j=1}^{K_i} X_{ij}$,
$X_{ij}$ is 
the 
measurement of item $j$ sampled from 
stratum $i$. 


 
\section{ Problem formulation and various approaches.}

Let $(X_i, K_i)$, $i=1,...,n$, be independent random 
vectors, where the conditional distribution of $X_i$ 
condition on $K_i$ is $B(K_i,p_i)$. 
It is desired to estimate  $${p}= \frac{\sum_i p_i}{n},$$ 
based on the observed $(X_i,K_i)$, $i=1,...,n$.

Given a population, we think of $n$ disjoint and 
exhaustive strata,  where $X_i$ is the number of  (say) 
unemployed in a random sample of size $K_i$ from
stratum $i, \; i=1,...,n$. 



When the strata are of equal size,  
${p}$ is the proportion of unemployed  in the population.  

\bigskip

{\bf Difficulty}

The difficulty in the above estimation problem is that,
the obvious estimator:
 $$\frac{1}{n} \sum_{i=1}^n  \frac{X_i}{K_i},$$
is not defined in case $K_i=0$ for some $i$.

In such cases there are ad-hoc approaches known as 
`collapsed strata', 
where small or empty strata 
are 
unified, after 
the data is observed.  See, e.g.,  
Wolter (1985), Chapter 2; there, the emphasis is on 
estimating the variance of estimators.

\bigskip

An `extreme collapsing' is to a single stratum, 
(which is, in fact, desirable when $p_1=...=p_n$). It yields 
the Extreme--Collapsing estimator:

$$\frac{ \sum X_i}{\sum K_i}.$$

In a non-response setup, under uniform sampling
and strata of equal size, the later estimator is 
natural  when individuals are  `Missing Completely 
At Random' (MCAR).

\bigskip

{\bf Naive estimator.}

Suppose that the strata that "`did not respond"' (i.e., with $K_i=0$),
are "`Missing Completely At Random"'. Formally, for $I$ 
uniform on $\{1,...,n \}$,
$E(p_I|K_I >0)=E(p_I) $, consequently, $E \frac{1}{m} \sum_{\{i|K_i>0\}} p_i= \frac{1}{n} \sum p_i$. Then, 
a reasonable   unbiased estimator is:

\begin{equation} \label{eqn:naive} Naive =\frac{1}{m} \sum_{ \{i| K_i>0 \} }  \frac{X_i}{K_i},  \end{equation} 
where $m=\#\{i|K_i>0\}$.  Assume $m>0$ w.p.1, in order to avoid formal difficulties
in the above.

\bigskip




However, in typical cases, strata are not missing 
completely at random, e.g., as mentioned, strata with higher rates
of 
unemployment are known to have lower response rates. Hence their corresponding 
$K_i$ are smaller and those strata are more likely to be missing, or, equivalently, 
more likely to have $K_i=0$. 


Note: when $K_i \equiv K >0$, the Extreme--Collapsing and the Naive estimators are identical.

\subsection {Our GMLE Estimator}

We consider a setup where  $(X_i,K_i) \sim F_{\theta_i}$, 
for a latent/unobserved $\theta_i$, and a known 
parametric family $\{ F_\theta , \; \theta \in \Omega \} $.

The joint distribution of $(X_i,K_i,\theta_i)$ is denoted $G^*$.
The marginal of $\theta_i$ is denoted $G$.
Under our Non-Parametric approach $G$ is completely 
unknown.

We consider an auxiliary  setup where $(X_i,K_i,\theta_i) \sim G^*$ are iid,
as in Empirical-Bayes.

In Scenario i),  $\theta_i=(\theta_{i1},\theta_{i2}) \equiv (\pi_i,p_i)$,
where $\pi_i$ is the probability of response, while $p_i$ is the 
proportion  (of, say, unemployed)  in stratum $i$.
Conditional on $\theta_i$, $X_i \sim B(K_i,p_i)$ and $K_i \sim B(\kappa_i, \pi_i)$.

In Scenario ii),  $\theta_i=(\theta_{i1},\theta_{i2}) \equiv (\lambda_i, p_i)$.
Conditional on $\theta_i$, $X_i \sim B(K_i,p_i)$  as before, 
and $K_i \sim Poisson(\lambda_i)$.



\bigskip

Denote $Y_i=(X_i,K_i)$.
Denote by $f_{\theta_i}(y) \equiv f(y|\theta_i)$ the conditional density of $Y_i=(X_i,K_i)$
conditional on $\theta_i$.

Given $\tilde{G}$, let $ f_{\tilde{G}}(y)= \int f_\theta(y) d\tilde{G}(\theta)$.
Note,  $Y_i$ are iid with density $f_G$.

The GMLE $\hat{G}$ (Kiefer and Wolfowitz (1956) ) is:
$$\hat{G}= argmax_{\tilde{G}} \; \Pi f_{\tilde{G}}(Y_i).$$

Traditionally the GMLE is computed via EM-algorithm, and this is also our method of computation in the current paper. 
Recently, Koenker and Mizera (2014) suggested 
exploitation 
of convex optimization techniques. 
See also the quadratic programming approach
of 
Greenshtein and Itskov (2018). 

\bigskip

{\bf Our estimator. }

In our motivating examples $\theta_i \equiv (\theta_{i1},\theta_{i2}) \equiv (\theta_{i1},p_i)$.

The distribution of $\theta_i$ is like a r.v. $\Theta=(\Theta_1,\Theta_2) \sim G$.
Estimate $E_G \Theta$ via $$E_{\hat{G}} \Theta.$$

{\bf Remark:} In case that $\hat{G}$ is not unique, we may consider our estimator as a
`set' of all estimates that correspond to all GMLE; alternatively, as we do in the sequel, as the estimator $E_{\hat{G}} \Theta$ that corresponds to the unique point $\hat{G}$ that the optimization procedure
(in this paper, the EM algorithm), converges to.

\bigskip

Note, in both scenarios, our target is: $$p= E_G \Theta_2.$$  
Note further: under the empirical Bayes formulation 
we replace the prediction of the (random) quantity 
$\frac{1}{n} \sum p_i$  by the estimation of $E_{G^*} \Theta_2$. 
The two quantities are  asymptotically equal.
On the relation between the two tasks of prediction and estimation see
Zhang (2005).

{\it Consistency of our estimator.}
In  Scenario ii)   it may be shown that,
under mild conditions, 
$\hat{G}$ is consistent for $G$  (i.e., weakly converges to $G$). 
Thus, $E_{\hat{G}} \Theta$ is consistent for $E_G \Theta$.
See Theorem \ref{thm:con}. and the following corollary.

In Scenario i)   $\hat{G}$ is 
{\it not} consistent for $G$.
This may be seen, since $Y_i=(X_i,K_i)$ has only finitely
many, $M$, different potential outcomes and 
identifiability becomes an issue.  In particular, 
regardless of $G$, there exists a 
GMLE $\hat{G}$, 
supported on
$M+1$ points at most, 
see Teicher (1963), 
Lindsay (1985).
See also, Example  \ref{ex:bin}.

\bigskip

Although in Scenario i), we do not have consistency,
our estimator, $E_{\hat{G}} \Theta$, still has an intuitive appeal
and surprisingly good performance in examples, 
it will be 
explored in simulations, along with Scenario ii).

Finding GMLE for a two dimensional distribution, as we do, only recently appears 
in the literature, see,
e.g., Gu and Koenker (2017), Feng and Dicker (2018). One reason might be its recent popularity
due to Koenker and Mizera`s computational methods.

\section{Simulations.}

In all of the following simulations, for convenience, 
the true $p$ equals 0.5. 
For any parameter configuration, the number of 
repetitions is 50.

The GMLE was computed via EM algorithm on a 
grid. The grids for $(\theta_{i1},\theta_{i2})$ contain 
$40 \times 40=1600$ grid points in a range that fits the relevant problem.
The parametrization in Scenario ii) is via a two dimensional Poisson, as explained in 
Section 5.

The number of iterations of the EM algorithm is 1000, 
and the last one is taken as the GMLE. The initial `guess' for $G$
is uniform on the grid points.

\bigskip

\subsection{ Poisson sample sizes.}




There are two types, Type I and II, of strata, 
500 of each type.
We report on the mean of the Naive and the GMLE 
estimators for $p$, and their corresponding sample 
standard deviation,  based on 50 repetitions for 
each 
configuration.


There are three configurations. In each configuration the corresponding $G$
is discrete having two points
support, corresponding to Type I and II strata.

Type I have corresponding  $\theta_i \equiv (\lambda_0^I,p_0^I) \equiv \theta^{I}$,
Type II  have corresponding  $\theta_i \equiv (\lambda_0^{II},p_0^{II}) \equiv \theta^{II} $.


\bigskip

In the following Table 1, in boldface is the average 
of 50 estimates, and in parenthesis  is the sample 
standard deviation of those estimates.

\begin{table} \label{tab:discrete}
\caption{Poisson Simulation discrete $G$.}
\begin{center}
\begin{tabular}{ |c||c|c| } 
 \hline 
    $\theta^{I},\theta^{II}$ & Naive & GMLE \\
	\hline \hline	
 $(2,0.4),(1,0.6)$ & {\bf 0.486}, (0.014) & {\bf 0.503}, (0.020) \\ 
 $(2,0.2),(1,0.8)$  & {\bf 0.453}, (0.015) & {\bf 0.496}, (0.018) \\ 
 $(2,0.2),(0.5,0.8)$ & {\bf 0.385}, (0.013) & {\bf 0.505}, (0.022) \\ 
 \hline
\end{tabular}
\end{center}
\end{table}

We also simulated more complicated scenarios, where the support of $G$
is continuous.  No, additional remarkable insights were discovered.
In the following we report on one such additional simulations.

\bigskip

{\bf Continuous ${\bf G}$ .}

In the simulations, presented in Table 2 the support of $G$ 
is continuous.
We have two types of strata, I and II, 500 of each type.
The distributions of $\theta^I$ and $\theta^{II}$ 
that correspond to the two strata are as
follows.

Type I: $\lambda_i \sim  {\mbox Uniform}(0.5,1)$, 
while the corresponding $p_i$ for strata $i$ of 
Type I is fixed  $p_i=p^I$.

Type II:  $\lambda_i \sim  {\mbox Uniform}(0.5,2)$,  
while the corresponding $p_i$ for strata $i$ of 
Type II is fixed  $p_i=p^{II}$.

In Table 2 we summarize three 
configurations,
where $p^{II}=(1-p^I$), $p^I$=0.4, 0.3, 0.2.

\bigskip 

\begin{table}
\label {tab:cont}
\caption{Poisson simulation with continuous $G$.}
\begin{center}
\begin{tabular}{ |c||c|c| } 
 \hline 
    $p^{I}$ & Naive & GMLE \\
\hline \hline	
 $0.4$ & {\bf 0.513}, (0.015) & {\bf 0.500}, (0.023) \\ 
 $0.3$ & {\bf 0.529}, (0.017) & {\bf 0.501}, (0.027) \\ 
 $0.2$ & {\bf 0.538}, (0.016) & {\bf 0.491}, (0.027) \\ 
 \hline
\end{tabular}
\end{center}

\end{table}

\subsection{ Binomial sample sizes.}

We study Scenario i), where $K_i$, the realized sample size from stratum $i$, is 
distributed $B(\kappa_i,\pi_i)$.

\bigskip

Again, our simulated populations have two types of strata, 
500 of each type.

In the simulations reported in Table 3,
Type I strata have $\pi_i=p_i \equiv p^I$,
similarly, Type II strata have $\pi_i=p_i \equiv p^{II}$.

In this case, the distribution $G$ of $\theta_i=(\pi_i,p_i)$,
is discrete, having two points support.
We summarize three 
configurations, where $p^I=(1-p^{II})$.

In all cases $\kappa_i \equiv 4$. The corresponding $p$
equals 0.5 throughout.

\begin{table} \label{tab:Binom1}
\caption{Binomial Simulation. $\kappa \equiv 4$}\label{tab2}
\begin{center}
\begin{tabular}{ |c||c|c| } 
 \hline 
    $\p^I,\p^{II}$ & Naive & GMLE \\
	\hline \hline	
 $0.2,0.8$ & {\bf 0.559}, (0.012) & {\bf 0.502}, (0.014) \\ 
 $0.3,0.7$ & {\bf 0.522}, (0.011) & {\bf 0.504}, (0.012) \\ 
 $0.4,0.6$ & {\bf 0.504}, (0.010) & {\bf 0.501}, (0.010) \\ 
 \hline
\end{tabular}
\end{center}
\end{table}
\bigskip

{\bf Continuous $G$.}

In the following simulations, reported in Table 4, the distribution $G$
of $\theta_i=(\pi_i,p_i)$ is continuous. For the same
$G$, we examine the values $\kappa_i \equiv \kappa=1,2,3,4,5.$

Type I strata have $p^I \sim \pi^I \sim U(0.1,0.6)$, 
Type II strata have $p^{II} \sim \pi^{II} \sim U(0.4,0.9)$; 
here, $p^k$ and $\pi^k$ are independent, $k=I,II$.
There are 500 strata of each type.


In all cases, again, $p=0.5$

\begin{table} \label{tab:Binomial2}
\caption{Binomial Simulations. $\kappa$=1,2,3,4,5.}\label{tab3}
\begin{center}
\begin{tabular}{ |c||c|c| } 
 \hline 
    $\kappa$ & Naive & GMLE \\
	\hline \hline	
 $1$ & {\bf 0.544}, (0.019) & {\bf 0.530}, (0.015) \\ 
 $2$ & {\bf 0.528}, (0.014) & {\bf 0.502}, (0.021) \\ 
 $3$ & {\bf 0.522}, (0.014) & {\bf 0.498}, (0.022) \\ 
 $4$ & {\bf 0.517}, (0.012) & {\bf 0.499}, (0.020) \\
 $5$ & {\bf 0.512}, (0.009) & {\bf 0.501}, (0.013) \\
 \hline
\end{tabular}
\end{center}
\end{table}
\bigskip

It is surprising how well the GMLE is doing
already for $\kappa=2,3$, in spite of the 
non-identifiability of $G$ and the
inconsistency of the GMLE. See also, Example \ref{ex:bin}.

\section{   Real Data Example.}

A rough description of the `Social-Survey', conducted 
yearly by the Israeli census bureau, is the following.  
We randomly sample a $1/1000$  fraction of the 
individuals in the registry, then verify  their home 
address and interview them in person.

We study real  `social survey' data,  accumulated for
Tel-Aviv, in the surveys  collected  in the years 
2015, 2016, 2017. The total sample size  in the 
three years is 1256. 
There are 156 `statistical-areas' in Tel-Aviv, very 
roughly of equal size, about 3000 individuals in each.  
Statistical-Areas are considered homogeneous in many 
respects, and we take them as strata.

Let $K_i$ be the  sample size in stratum $i$, $i=1,...156$, 
then our data satisfy $K_i>0$, $i=1,...,156$.
( In fact, we neglect a few small strata that actually 
had zero sample sizes ).

Let $p_i$ be the proportion of individuals in stratum $i$,
that own their living place.
The goal is to estimate $\frac{1}{n} \sum p_i$.
When the strata are of equal size, the later is the 
proportion of individuals that own their living-place  
(or, owned by a member of their household). 
The survey is of  individuals whose age is 20 or more.

The estimated proportion (per the three 
years) obtained by the naive/(obvious) estimator is: 
$$\frac{1}{n} \sum_{i=1}^n  \frac{X_i}{K_i}=
{\bf  0.434}.$$ 
Note, the naive estimator is applicable,
 since $K_i>0$, $i=1,...,156$.  Note further, that 
in this case, the naive
estimator is GMLE, as shown in corollary \ref{cor:naive}.

The `extreme collapse' 
estimator, satisfy: $$\frac{\sum X_i}{\sum K_i}={\bf 0.488}.$$

The significant difference between the two estimates 
has to do   {\it also} with the fact that strata are, in fact, not of 
equal size. But, more  importantly for us, the 
`extreme collapse'  seems  to over estimate the 
proportion, since owners are over represented in 
the sample. One reason is that their address in 
the registry is more accurate and thus it is easier to 
find them. In other words individuals are not MCAR. 
This phenomena is partially corrected by the 
stratification, when MAR conditional on the fine strata 
is approximately right.

\bigskip

In the following we simulated  scenarios in which  only 
a portion $\gamma$ of the described sample,  would have been attempted to 
be sampled. The results of 25 simulations applied on 
the real data  with $\gamma=0.1, 0.2, 0.25$, are presented 
in the following  Table 5.  The (random) number of 
simulated strata with zero sample sizes, corresponding 
to $\gamma=0.1, 0.2, 0.25$,  are around 70, 40, and 30, 
correspondingly.

Both, for the entire data  and the simulations, the 
sample  sizes $K_i$ are modeled as ${\mbox Poisson}(\lambda_i)$, $i=1,...,156$. 
The naive estimate based on  the entire data equals {\bf 0.434},
which is a reasonable benchmark.

\bigskip

\begin{table}
\caption{ Real Data.}\label{tab4}
\begin{center}
\begin{tabular}{ |c||c|c| } 
 \hline 
    $\gamma$ & Naive & GMLE \\
	\hline \hline	
 $0.1$ & 0.471 & 0.457 \\ 
 $0.2$  & 0.467 & 0.443 \\ 
 $0.25$ & 0.447 & 0.434 \\ 
 \hline
\end{tabular}
\end{center}
\end{table}
\bigskip

\section{Theoretical and asymptotic Results.}

Some of the results in this section are under a general formulation, beyond Scenarios i) and ii). In addition, in order to simplify certain formal considerations, we assume the following:

\begin{Assumption}
The support  of the parameter space is bounded.
\end{Assumption}

This assumption simplifies, e.g.,  verifying various conditions in Kiefer and Wolfowitz (1956),
it also implies compactness of any class $\{ G \}$ of distributions on the closure of the parameter space, since that any sequence of distributions on the closure is tight.


\subsection{ Scenario ii) } 

In Scenario ii) there is consistency in the estimation of  $E_G \Theta$,
when $P_G(\lambda=0) \equiv G(\{ \lambda=0 \})=0$,
as proved in the following theorem.

We first remind the notion of identifiability of mixtures.
Let $f_G(x)=\int f_\theta(x) dG(\theta)$. Then, $G$ is identifiable
if there exists no $\tilde{G}$ such that $f_G(x)=f_{\tilde{G}}(x)$ for every $x$.

 It is convenient to re-parametrize
the problem, as follows. Given $X_i \sim B(K_i, p_i)$ conditional on $K_i$,
and $K_i \sim Poisson(\lambda_i)$.
Denote $W_{i1} \equiv X_i$, and $W_{i2}=K_i- X_i$. Then $W_{i1}$ and
$W_{i2}$ are  independent  Poissons conditional on $(\lambda_i,p_i)$, with corresponding parameters 
$\xi_{1i} \equiv p_i \lambda_i$
and $\xi_{2i} \equiv (1-p_i)\lambda_i$. 

\begin{theorem}
Consider $G_\xi$ the distribution of $(\xi_1,\xi_2)$ as above.
Let $\hat{G}_\xi$ be the GMLE based on iid $(W_{i1}, W_{i2}), \; i=1,...,n$ from the mixture 
$G$.
Then $\hat{G}_\xi$ converges weakly to $G_\xi$. 
\end{theorem} \label{thm:con}
\begin{proof}
The proof follows from Kiefer and Wolfowitz (1956). Checking the conditions is standard,
the identifiability condition is verified, e.g., by
Dimitris Karlis and Evdokia Xekalak (2005).
\end{proof}

\begin {corollary}
In Scenario ii) for any  function $\eta$, such that $\eta(\theta)=\psi(\xi_1,\xi_2)$,
for  $\psi$ which is continuous  and bounded on the support of $(\xi_1,\xi_2)$ under $G$, 
$E_{\hat{G}} \eta(\Theta) \rightarrow E_G \eta(\Theta)$.
\end{corollary}

In our setup $\eta_\epsilon(\theta)=\psi_\epsilon(\xi_1,\xi_2) 
\equiv\frac{\xi_1}{\xi_1 + \xi_2} 1\{\xi_1+\xi_2>0\}$ satisfies
the requirements   for any $\epsilon>0$. If $G( \{ \lambda=0 \})=0$
then $\lim_{\epsilon \rightarrow 0} E_G \eta_\epsilon= E_G \Theta_2$, and $E_{\hat{G}} \Theta_2 \rightarrow E_G \Theta_2$.



\subsubsection{Persistence} \label{rate}
 When confining ourselves  to the class of distributions in 
$\Gamma= \{ G | G(\{\lambda=0 \})=0 \}$,
we have consistency of the estimator $E_{\hat{G}} \Theta_2$ for $E_G \Theta_2$ for
any $G \in \Gamma$. 
However, the convergence may be arbitrarily slow. In particular there
is no {\it persistence}, specifically we may find a sequence of distributions $G_n$, $G_n \in \Gamma$
such that  for some $\epsilon >0$, $P_{G_n} (|\hat{p}^n- p^n|>\epsilon) \not\rightarrow 0$,
where $\hat{p}^n =E_{\hat{G}_n} \Theta_2$, $\hat{G}_n$ is a GMLE, and $p^n=E_{G_n} \Theta_2$. 

Persistence is implied, e.g., when we narrow ourselves to the collection $\Gamma'$ of all distributions $G$ whose support is uniformly bounded away from $\{ \lambda=0 \}$. This, follows by utilizing
the compactness of $\Gamma'$.

The slow rate in the estimation of $E_G (\Theta_2)$, stems from the 
slow rate in the estimation of $E_G(\Theta_2|K=0)$. If $P_{G_n}( K>0)> \epsilon_0$,  for some $\epsilon_0>0$, then the estimation of $E_{G_n}(\Theta_2|K>0)$ and
$P_{G_n}(K>0)$ may be done in $\sqrt{n}-$rate.

For a fixed $G$, whose support does not include points with $\lambda=0$,  we do not know the rate at which $E_{\hat{G}} \Theta_2$ approaches $E_G \Theta_2$. A more challenging problem is
the optimality of $\hat{p}^n$ in the following sense. Is there a sequence $\tilde{p}^n$ such that for any sequence $G_n \in \Gamma$, if $\hat{p}^n -p^n \rightarrow_{G_n} 0$ then
$\tilde{p}^n -p^n \rightarrow_{G_n} 0$, but the converse does not hold? here, the convergence is in probability.


The theoretical Exploration of persistence is in order 
to understand how to avoid introducing too delicate stratification,
where $p^n$ can not be `decently' estimated. In practice, one could apply 
the 'non-standard cross-validation' suggested in Brown, et.al. (2013), in order to determine
whether the estimation of $p^n$, for a specific stratification and a corresponding  $n$, is
too `ambitious' and `unreliable'.
Another way to evaluate  a specific stratification is the method suggested in Sub-Section \ref{subsec:bounding}, of obtaining  bounds for $p^n$ based on the sample.



\subsection{Scenario i) } 

In Scenario i) there is no consistency in estimating $G$, neither a consistency in
estimating $E_G(\Theta_2)$. This is a result of the non-identifiability, and it is demonstrated in
the following example for   $\kappa \equiv 1$.

\begin{example} \label{ex:bin}
When $\kappa=1$ there are $M=3$ possible outcomes of the $i'th$ observation, we list them as: $X_i=1$,  $X_i=0$, and $K_i=0$; the corresponding probabilities are: 
$(\pi_i p_i, \pi_i(1-p_i), (1-\pi_i))$. 
Suppose the outcomes of $n$ realization have $n_1$ occurrences of $X_i=1$,  $n_2$ occurrences of $X_i=0$, and $n_3$ occurrences of $K_i=0$. Note $(n_1,n_2,n_3)$ is multinomial.
Suppose $\frac{1}{n}(n_1,n_2,n_3)=(0.25,0.25,0.5)$, obviously the following $\hat{G}_1$ and 
$\hat{G}_2$
are both GMLE. Let $\hat{G}_1$ be degenerate at $(\pi,p)=(0.5,0.5)$. Let $\hat{G}_2$ be the distribution whose
support is $(0,1), (1,0), (1,1)$ with corresponding probabilities 0.5, 0.25, 0.25.
Then obviously both $\hat{G}_1$ and $\hat{G}_2$ are GMLE, while $E_{\hat{G}_1} \Theta_2 =0.5 \neq E_{\hat{G}_2} \Theta_2 = 0.75.$ 
\end{example}



\subsection{ Alternative estimators and representations.}

Given a function $\eta$, suppose it is desired to estimate
$E_G \eta(\Theta)$.
By the  following theorem, the estimator $\hat{\eta}= E_{\hat{G}} \eta( \Theta)$ equals to the average of 
$E_{\hat{G}}(\eta(\Theta) |   Y_i)$, $i=1,...,n$. The appeal of this fact is that if we estimate
 the values of the individual $\eta(\Theta_i)$ via non-parametric empirical Bayes under squared loss, specifically by 
$E_{\hat{G}} (\eta(\Theta )| Y_i)$, there is a consistency and agreement between the estimates of the individual parameters and the estimates of their total, or, of their average.
In Zhang (2005), the problem of estimating random sums involving a latent variable is explored, in particular 
the existence of $\sqrt{n}$ consistent estimators and their efficiency. One approach in 
Zhang (2005) is to estimate the random sum, by the sum of the estimated  conditional expectations of the summands. This is analogous to estimate  $\sum_i \eta(\Theta_i)$ by
$\sum_i E_{\hat{G}} (\eta(\Theta) | Y_i)$, the last term equals to $nE_{\hat{G}} \eta(\Theta)$ by the following theorem.

\begin{theorem} \label{thm:agreement}
Assume $\eta$ is a bounded function, then:
$$\hat{\eta} \equiv E_{\hat{G}} \eta(\Theta) = \frac{1}{n}\sum _i E_{\hat{G}}(\eta(\Theta )| Y_i).$$
\end{theorem}

\begin{proof}

In the following proof we give two independent arguments in two parts. Part ii) is general, while Part i)
is not, since it assumes convergence of the EM-algorithm. Part i)  is brought in order to give further perspective, and to specify the EM-algorithm which is used in this paper.

i) Given a realization $Y_1,...,Y_n$, assume that the EM-algorithm converges to a unique maximum $\hat{G}$.
The $(k+1)'th$ iteration
is related to the $k'th$ iteration via 
$$d\hat{G}^{k+1}(\theta_0)= \frac{1}{n}\sum_i \frac{ f(Y_i|\theta_0)d\hat{G}^{k}(\theta_0)}
{ \int  f(Y_i|\theta)d\hat{G}^{k}(\theta)}.$$
For $k=\infty$, when the GMLE $\hat{G} \equiv \hat{G}^{\infty}$,  is plugged into the above equality,
the proof follows when taking the expectation of $\eta(\Theta)$ under both representations, and
interchanging the order of summation and integration  that correspond to the right hand side.

ii) For  a given GMLE $\hat{G}$, a function $\eta$, and $\hat{\eta}=E_{\hat{G}} \eta(\Theta)$, let 
$h(\theta)=\eta(\theta) -\hat{\eta}$; note, $E_{\hat{G}} h(\Theta)=0$. 
Let $$d\hat{G}_t(\theta)= (1+t*h(\theta))d\hat{G}(\theta).$$ 
Since $\eta$ is bounded, for a small enough $t_0$ we may
consider the class of distributions $\{ G_t, \; t \in (-t_0,t_0) \}$.

Since $\hat{G}$ is a GMLE  we have:
\begin{eqnarray*}
0&=& \frac{d}{dt} \sum_i  \log (\int f(Y_i|\theta) d\hat{G}_t(\theta) ) |_{t=0}\\
&=&\sum_i \frac { \int h(\theta) f(Y_i|\theta) d\hat{G}(\theta) }
{\int f(Y_i|\theta) d\hat{G}(\theta)}.
\end{eqnarray*}
Hence,
$\frac{1}{n}\sum_i \int \eta(\theta) d\hat{G}(\theta|Y_i)= \hat{\eta}.$

\end{proof}

\bigskip

Consider $Y_i=(X_i,K_i)$, as in scenarios i) and ii), and let

\[ \Psi_{\hat{G}}^*(Y) \equiv \Psi_{\hat{G}}^*  (X,K) = \left \{ \begin{array}{ll}
\frac{X}{K} \;\;  & \; K>0 \\
\\
E_{\hat{G}}(\Theta_2|K=0) \;\;\;       &   \; K=0\\
\end{array}
\right. \]

We write $\Psi_{\hat{G}}^* \equiv \Psi^*$. 

The estimator 
\begin{equation} \label{eqn:appeal} 
\tilde{\eta}^*=\frac{1}{n} \sum_i \Psi^*(Y_i), 
\end{equation}
 also has an appeal.
In our simulations the estimators $\tilde{\eta}^*$ and $\hat{\eta}$ are nearly equal.
For example in the 50 simulations described in Table 2, for
the cases  $p^I=0.2, 0.3, 0.4$,
the corresponding averages of the absolute differences  $|\tilde{\eta}_j^*-\hat{\eta}_j|, \;j=1,...,50,$  are
$0.00075, 0.0011, 0.0011$.

From Sub-Section \ref{sec:expfam} it follows that in fact $\hat{\eta}$ and $\tilde{\eta}^*$
are equal.

\subsubsection  { GMLE for the mean of a mixture of an exponential family} 
\label{sec:expfam}

Let $Y_i \sim N(\theta_i,1), \; i=1,...,n$ be independent observations, the obvious estimator for  $\sum \theta_i$ is $\sum Y_i$. 
One may wonder about a comparison between the obvious estimator $\sum Y_i$ and the estimator $nE_{\hat{G}} \Theta$. 
Similarly on may wonder about the analogous situation where $Y_i \sim Poisson(\lambda_i)$,
and it is desired to estimate $\sum \lambda_i$.
In the following we will show that the
two estimators are in fact equal. Thus, "`GMLE sophistication"' does not improve on the obvious estimator, yet, it does not  harm.

The following theorem covers a general situation. The notations are for a one dimensional
exponential family, but it applies for a general exponential family. 

\begin{theorem} \label{thm:expfam} Let $Y_1,...,Y_n$, be independent observations 
$Y_i \sim f_{\theta_i} (y)$.
Suppose that the density of $Y$ is of the form $f_\theta(y) \equiv f(y|\theta)= \exp(\theta y-\psi(\theta))$, $\theta \in \Omega$,
with respect to some dominating measure $\mu$. 
Let $\hat{G}$ be a GMLE whose support is on interior points of the parameter set $\Omega$.

Then: $$nE_{\hat{G}} \eta(\Theta)= \sum Y_i.$$
\end{theorem}
\begin{proof}
Given the observations $Y_1,...,Y_n$, let $\hat{G}$ be a GMLE, define
the translation $\hat{G}_\Delta=\hat{G}+ \Delta$.

Note: $$ \int \exp(\theta y -\psi(\theta)) d \hat{G}_\Delta(\theta)= 
 \int \exp((\theta-\Delta) y -\psi(\theta-\Delta)) d \hat{G}(\theta).$$

Since $\hat{G}$ is GMLE,
\begin{eqnarray*}
0&=& \frac{d}{d\Delta} \sum \log(\int \exp(\theta Y_i -\psi(\theta)) 
d \hat{G}_\Delta(\theta) ) |_{\Delta=0} \\
&=&  \frac{d}{d\Delta} 
\sum \log ( \int \exp((\theta-\Delta) Y_i -\psi(\theta-\Delta) ) d \hat{G}(\theta) ) |_{\Delta=0}\\
&=& \sum  \frac{  \int (-Y_i + \psi'(\theta) )  \exp(\theta Y_i -\psi(\theta)) d \hat{G}(\theta) }  
{\int \exp(\theta Y_i -\psi(\theta)) d \hat{G}(\theta)}\\
&=& -\sum Y_i + \sum E_{\hat{G}} (\psi'(\Theta) | Y_i)\\
&=& -\sum Y_i + nE_{\hat{G}} \eta(\Theta)
\end{eqnarray*}

The last equality follows by Theorem \ref{thm:agreement}, and since  for exponential family
$\eta(\theta)=\psi'(\theta)=E_\theta Y$.

\end{proof}

The above theorem implies that under an exponential family setup, given two GMLE
$\hat{G}_1$ and $\hat{G}_2$, we have $E_{\hat{G}_1} \Theta= E_{\hat{G}_2} \Theta$.

The above theorem may be extended to cover higher moments of the canonical variable.

\begin{corollary} \label{cor:naive}
In scenarios i) and ii), suppose $K_i>0, \; i=1,...,n$. Then the naive estimator
$\frac{1}{n}\sum \frac{X_i}{K_i}$ is GMLE for $p$.
\end{corollary}

\begin{proof}
We should show that $E_{\hat{G}} \Theta_2= \frac{1}{n}\sum \frac{X_i}{K_i}$.
We elaborate on scenario ii).  Note that conditional on $(\lambda_i,p_i)$,
$(K_i, X_i)$ may be considered as canonical variables of a member of a two dimensional exponential family
with canonical parameters $(\log(\lambda_i), \log( p_i/(1-p_i) )$. Transforming
the variables to $(K_i,\frac{X_i}{K_i}) \equiv (W_{i1},W_{i2}) \equiv W_i$,  defines
new canonical variables in a corresponding exponential family, with a corresponding distribution $G_\xi$ of the induced new canonical parameters, denoted 
$\xi_i \equiv (\xi_{i1},\xi_{i2})$; a corresponding $\psi(\xi_{i1},\xi_{i2}) \equiv  \psi(\xi_i)$ is also defined, where $\psi'(\xi_i)= E_{\xi_i} W_i$.
By Theorem \ref{thm:expfam}, $E_{\hat{G}_\xi } \psi' (\xi)=\frac{1}{n} \sum W_i$.
\end{proof}

\bigskip

A main conclusion of the above development is that the estimators $\tilde{\eta}^*$,
and $\hat{\eta}$, are equal.

\begin{corollary}
Let $\hat{G}$ be GMLE under scenarios i) and ii). Suppose the support of $\hat{G}$ is 
is on interior points of the parameter space. Then $\tilde{\eta}^*=\hat{\eta}$.
\end{corollary}

\begin{proof}
Consider the representation in Corollary  \ref{cor:naive}, with $W_i=(W_{i1},W_{i2})$
and $\xi_i=(\xi_{i1},\xi_{i2})$; when $W_{i1}=0$ we define $W_{i2} \equiv NULL$. For observations $W_i$ with $W_{i1}>0$ the density $f(W_i|\xi,W_{i1}>0)$ has an exponential
family form with a corresponding $\psi(\xi_1,\xi_2)$; for $W_i$ with $W_{i1}=0$, 
$f(W_i|\xi)=P_\xi( W_{i1}=0)$.
Let $p_i=\eta(\xi_i)=E_\xi(W_{i2}| W_{i1}>0)$.
Then, by Theorem \ref{thm:agreement} $$E_{\hat{G}}\eta(\xi)= \sum_i E_{\hat{G}} ( \eta(\xi)| W_i)=
\sum_{W_{i1}=0} E_{\hat{G}} ( \eta(\xi)| W_i) + \sum_{W_{i1} > 0} E_{\hat{G}} ( \eta(\xi)| W_i).$$
In order to prove the claim we need to show that 
$$ \sum_{W_{i1}>0}  E_{\hat{G}} ( \eta(\xi)| W_i)= \sum_{W_{i1}>0} W_{i2}  .$$
In order to show the later we use an argument similar to that in 
Theorem \ref{thm:expfam}.
Let $\Delta=(0,\Delta_1)$ and $\hat{G}_\Delta=\hat{G}+\Delta$.
Since $\hat{G}$ is a maximal point we have the following for the  derivative.
\begin{eqnarray*}
 0 &=& \frac{d}{d \Delta_1} \Bigl(  
  \sum_{W_{i1}=0} \log(\int f(W_i|\xi) ) d\hat{G}(\xi)\\
	&+& \sum_{W_{i1}>0} \log(\int P_\xi (W_{i1}>0) f(W_i|\xi, W_{i1}>0) ) d\hat{G}(\xi)
								\; \Bigl)|_{\Delta_1=0} \\
  &=& 0 \; + \; \sum_{W_{i1}>0} \int (-W_{i1} + \eta(\xi) \; ) d\hat{G} (\xi | W_i)
\end{eqnarray*}
The last equality is obtained along the lines of Theorem  \ref{thm:expfam},
observing that $P_{\xi_i} (W_{i1}>0)$ does not depend on $\xi_{i2}$ and that 
$P_\xi(W_{i1}>0) f(W_i|\xi,W_{i1}>0)= f(W_i|\xi_i)$ when $W_{i2}>0$.
The assertion now follows.

\end{proof}

\subsection{ Bounding population's average.} \label{subsec:bounding} 

As demonstrated in Example 1, the GMLE may be inconsistent. We suggest a way
to bound the value of $E_G \Theta$,
via a confidence interval.
 The approach is related to Greenshtein and Itskov (2018).  It is presented  for cases where the number
of outcomes $M$ is finite, denoted $(U_1,...,U_M)$. Let $\hat{F}_n$ be the empirical distribution of $(U_1,...,U_M)$.  For every $G$ consider the convex set $\Gamma$ of all 
mixtures $G$ satisfying,
$ L(\hat{F}_n|Q_G)< \chi^2_{(M-1)} (1-\alpha)$. Here $L$ is the log-likelihood, and $Q_G$
is the distribution induced by the mixture $G$ on $U_i,\; i=1,...,M$.

Then, an asymptotic conservative level-$\alpha$ CI is the solution to the  convex problem:

$$(\min_{G \in \Gamma} E_G \Theta, \max_{G \in \Gamma} E_G \Theta ).$$
Here we use the asymptotic $\chi^2_{M-1}$ distribution of the log-likelihood.


An elaboration on the computation is beyond the scope of this paper, and beyond the `scope' of the authors.
In case of an infinite number of possible outcomes a related approach is to "`discretize"'
to $M$ possible outcomes. One may  want to allow $M=M_n$ but, a formal treatment 
of those ideas is, again, beyond the scope of this paper.

\section {Genralizations.} \label{sec:gen}

{\bf General Setup.} 
In each strata $i$, $i=1,...,n$, we 
observe
$K_i$ i.i.d  observations $X_{ij}, \; j=1,...,K_i$,  
with mean $\mu_i$. 

Suppose the proportion of strata $i$ in the
population 
is $a_i$, and it is desired to estimate $\sum a_i \mu_i$. 
As in the previous sections $K_i$ may equal zero.
We treat the data as iid vectors  $(K_i,\mu_i,a_i,X_{i1},X_{i2},....)$, 
$i=1,...,n$, 
from a distribution $G^*$, where we observe
$K_i$, $a_i$, and $(X_{i1},...,X_{iK_i})$. Hence, $\sum a_i \mu_i$ is
random and in fact we attempt to estimate $E_{G^*} \sum a_i \mu_i $. 

In case $K_i=0$,  
we define  $(X_{i1},...,X_{iK_i}) \equiv NULL$. 

As before $K_i$ is random. We stress
that there is no 
assumption on the joint distribution of  $K_i$, $a_i$ and $\mu_i$.






Let $Z_{ij}=a_i X_{ij}$, we confine ourselves to estimators 
which are functions of $(Z_{i1},...,Z_{ik_i})$,  $i=1,...,n$; 
again,
when $K_i=0$  we  observe "`NULL"'.
Note, when $K_i>0$ w.p.1, the natural ( predictor) estimator for 
 ($\sum a_i \mu_i$)  $E_{G^*} \sum a_i \mu_i$ is: 

$$\frac{1}{n}\sum_i \frac{\sum^{K_i}_j Z_{ij}}{K_i}.$$
Note further,  $$\mu \equiv E_{G^*} \sum a_i \mu_i= E_{G^*} Z_{11}=\int P_{G^*}( Z_{11}>c) dc .$$
For appropriate estimator $\hat{P}(Z_{11}>c)$  
we will estimate ${\mu} \equiv E_{G^*} \sum a_i \mu_i$  by 
 $$\hat{\mu} =\int \hat{P}( Z_{11}>c) dc .$$

Specifically, for every $c$  define  the indicator $ I_{ij}^c= I(Z_{ij}>c)$. 
Then, conditional on  $K_i$,
the following Binomial random variables are defined,
$\sum_j I_{ij}^c \sim B(K_i,p_i^c)$.

Analogously to the previous sections, where binary variables were treated,
we may estimate 
$P_{G^*}(Z_{11}>c) =E_{G^*} \frac{1}{n} \sum p_i^c \approx\frac{1}{n} \sum p_i^c $ by 
$\hat{p}^c$.

Now we may  apply the estimator:
$$\hat{\mu}=\int \hat{p}^c dc .$$

\bigskip

{\bf \Large References:}

\begin{list}{}{\setlength{\itemindent}{-1em}\setlength{\itemsep}{0.5em}}
\item
Brown L.D., Greenshtein, E. and Ritov, Y. (2013). The Poisson compound decision revisited. {\it JASA.} {\bf 108} 741-749.
\item
Dimitris Karlis and Evdokia Xekalak (2005), Mixed Poisson Distributions,
{\it International Statistical Review}, 73, 1, 35–58, Printed in Wales by Cambrian Printers.
\item
Greenshtein, E. and Ritov, Y. (2004). Persistence in high-dimensional linear predictor selection and the virtue of overparametrization. {\it Bernoulli}, Vol 10, Number 6 (2004), 971-988.
\item
Greenshtein, E. (2006). Best subset selection, persistence in high-dimensional statistical learning and optimization under l1 constraint. {\it Ann.Stat.}
Volume 34, Number 5, 2367-2386.
\item
Greenshtein, E. and Itskov, T (2018), Application of Non-Parametric  Empirical
Bayes to Treatment of Non-Response. {\it Statistica Sinica} 28 (2018), 2189-2208.
\item
Gu, J. and Koenker, R. (2017). Unobserved Heterogeneity in Income Dynamics: An Empirical Bayes Perspective. {\it Journal of Business and Economics Statistics}. Volume 35, 2017 - Issue 1
\item
Kiefer, J. and Wolfowitz, J. (1956). Consistency of the maximum likelihood estimator in the presence of infinitely many incidental parameters. {\it Ann.Math.Stat.}
27 No. 4, 887-906.
\item
Koenker, R. and Mizera, I. (2014). Convex optimization, shape constraints,
compound decisions and empirical Bayes rules. {\it JASA } 109, 674-685.
\item
Lindsay, B. G. (1995). Mixture Models: Theory, Geometry and Applications.
Hayward, CA, IMS.
\item
Little, R.J.A and Rubin, D.B. (2002). Statistical Analysis with Missing Data.
New York: Wiley
\item
Long, Feng and Lee, H. Dicker (2018).
Approximate nonparametric maximum likelihood for mixture models: A convex optimization approach to fitting arbitrary multivariate mixing distributions. {\it Computational Statistics \& Data Analysis} 122: 80-91.
\item
J.M. Robins and Y. Ritov. Toward a curse of dimensionality appropriated
(CODA) asymptotic theory for semi-parametric models. {\it Statistics in Medicine},
16:285–319, 1997.
\item
Rosenbaum, P. R. and Rubin, D. B. (1983). The central role of the propensity score in observational studies for causal effects. {\it Biometrika,} Vol 70, Issue 1, April 1983, Pages 41–55.
\item
Teicher, H. (1963). Identifiability of finite mixtures. {\it Ann. Math. Stat.},
{\bf 34}, No. 4, 1265-1269. 
\item
Wolter, K. M. (1985). Introduction to Variance estimation. Second edition, Springer.
\item
Zhang, C-H. (2005). Estimation of sums of random variables: Examples and information bounds. {\it Ann. Stat.} {\bf 33}, No.5. 2022-2041.

\end{list}
\end{document}